\documentclass[11pt]{amsart}

\usepackage{amsmath,amsfonts}
\usepackage{amsmath,amsfonts,latexsym}
\usepackage{amscd,amssymb,amsmath}
\usepackage{amsfonts}
\usepackage{amsbsy}
\usepackage{epsfig,afterpage}
\usepackage{psfrag}
\usepackage{graphicx}
\usepackage{wasysym}
\usepackage[all,cmtip]{xy}
\usepackage{subfigure,tikz}
\usepackage{multimedia}

\newcommand{\bea}{\begin{eqnarray*}}
\newcommand{\eea}{\end{eqnarray*}}
\newcommand{\zz}[1]{}

\newtheorem{theorem}{Theorem}[section]
\newtheorem{proposition}[theorem]{Proposition}
\newtheorem{corollary}[theorem]{Corollary}
\newtheorem{remark}[theorem]{Remark}

\newtheorem{lemma}[theorem]{Lemma}
\newtheorem{example}[theorem]{Example}

\newcommand{\bz}{\mathbb{Z}}

\newcommand{\UU}{\mathcal U}

 \newcommand{\NN}{{\mathbb{N}}}
 \newcommand{\ZZ}{{\mathbb{Z}}}
 \newcommand{\RR}{{\mathbb{R}}}
 \newcommand{\CC}{{\mathbb{C}}}
 \newcommand{\HH}{{\mathbb{H}}}

\newcommand{\sct}{{\rm sct}}
\newcommand{\ct}{{\rm ct}}
\newcommand{\cat}{{\rm cat}}
\newcommand{\hdim}{{\rm hdim}}
\newcommand{\rank}{{\rm rank}}

\begin{document}
\thanks{$^{**}$ Supported by the Research Grant NCN Grant  2015/19/B/ST1/01458 }
\thanks{$^{***}$ Supported by the Slovenian Research Agency program P1-0292 and grant J1-7025}
\title[Estimates of covering type]{Estimates of covering type and the number of vertices of 
minimal triangulations}

\author[Dejan Govc]{Dejan Govc$^*$}
\author[Wac{\l}aw Marzantowicz]{Wac{\l}aw Marzantowicz$^{**}$}
\author[Petar Pave\v{s}i\'{c}]{Petar Pave\v si\'{c}$^{***}$}

\address{$^*$, $^{***}$ Faculty of Mathematics and Physics, University of Ljubljana,
Jadranska 21,  1000 Ljubljana, Slovenija}
\email{dejan.govc@gmail.com, petar.pavesic@fmf.uni-lj.si}
\address{$^{**}$ \;Faculty of Mathematics and Computer Science, Adam Mickiewicz University 
of Pozna{\'n}, ul. Umultowska 87, 61-614 Pozna{\'n}, Poland.}
 \email{marzan@amu.edu.pl}

\maketitle

\begin{abstract} 
The \emph{covering type} of a space $X$ is a numerical homotopy invariant that in some sense 
measures the homotopical size of $X$. It was first introduced by Karoubi and Weibel 
\cite{K-W} as the minimal cardinality of a good cover 
of a space $Y$ taken among all spaces $Y$ that are homotopy equivalent to $X$. 
In this paper we give several estimates of the covering type in terms of other 
homotopy invariants of $X$, most notably the ranks of the homology groups of $X$, 
the multiplicative structure of the cohomology ring of $X$ and the Lusternik-Schnirelmann category 
of $X$. In addition, we relate the covering type of a triangulable space to the number of vertices 
in its minimal triangulations. In this way we derive within a unified framework several 
estimates of vertex-minimal triangulations which are either new or extensions of results that 
have been previously obtained by ad hoc combinatorial arguments. 
\end{abstract}

%%----------------------------------------------------------------------------------------------------------------------------------------------------
%% INTRODUCTION
%%----------------------------------------------------------------------------------------------------------------------------------------------------

\section{Introduction}

Many concepts in topology and homotopy theory are related to the size and the structure of 
the covers that a given space admits. Typical examples that spring in mind 
are Lebesgue dimension and Lusternik-Schnirelmann category. M. Karoubi and C. Weibel \cite{K-W} 
have recently 
introduced another interesting measure for the complexity of a space based on the size of its good
covers. 

Recall that an open  cover $\mathcal{U}$ of $X$ is said to be a \emph{good cover} if all elements
of $\mathcal{U}$ and all their  
non-empty finite intersections are contractible. 
Karoubi and Weibel defined $\sct(X)$, the \emph{strict covering type} of a given space $X$, 
as the minimal cardinality
of a good cover for $X$. Note that $\sct(X)$ can be infinite (e.g., if $X$ is an infinite 
discrete space) or even undefined, if the space
does not admit any good covers (e.g. the Hawaiian earring). In what follows we will always 
tacitly assume that the spaces under consideration
admit finite good covers. 

Strict covering type is a geometric notion and is not homotopy invariant, which led  
Karoubi and Weibel to define the \emph{covering type}
of $X$ as the minimal size of a good cover of spaces that are homotopy equivalent to $X$:
$$\ct(X):=\min\{\sct(Y)\mid Y\simeq X\}.$$
The covering type is a homotopy invariant of the space and is often strictly smaller than the 
strict covering type, even for 
simple spaces like wedges of circles (\cite[Example 1.3]{K-W}). 
Karoubi and Weibel also proved a useful result (\cite[Theorem 2.5]{K-W}) that the covering type 
of a finite CW complex is equal
to the minimal cardinality of a good \emph{closed} cover of some CW complex that is 
homotopy equivalent to $X$. Furthermore, they computed exactly the covering type for finite 
graphs (i.e., finite wedges of circles) 
and for some closed surfaces 
(the sphere, torus, projective space), while giving estimates for the covering type of other surfaces.
Finally, they estimated the covering type of mapping cones, suspensions and covering spaces. 
With our methods we have been able to refine and improve their results, and even to 
correct the upper estimate of $\ct(\RR P^m)$ in \cite[Example 7.3.]{K-W} (cf. Corollary 
\ref{cor:ct projective spaces}).

Good covers arise naturally in many situations, e.g., as geodesically convex neighbourhoods in Riemannian manifolds or as locally convex covers of polyhedra.
Their main feature is that the pattern of intersections of sets of a good cover capture the homotopy type of a space. Specifically, let $N(\UU)$ denote the \emph{nerve} of the open 
cover  $\UU$ of $X$, and let $|N(\UU)|$ be its geometric realization. We may identify the vertices of $|N(\UU)|$ with the elements of $\UU$ and the points of 
$|N(\UU)|$ with the convex combinations of elements of $\UU$.
If $\UU$ is \emph{numerable}, that is, if $\UU$ admits a subordinated partition of unity $\{\varphi_{_U}\colon X\to [0,1]\mid U\in\UU\}$, 
then the formula
$$\varphi(x):=\sum_{U\in\UU} \varphi_{_U}(x)\cdot U$$
determines the so called \emph{Aleksandroff map} $\varphi\colon X\to |N(\UU)|$, which has many remarkable properties.
In particular we have the following classical result, whose discovery is variously attributed to J.~Leray, K.~Borsuk and A.~Weil 
(see \cite[Corollary 4G.3]{Hatcher} for a modern proof).

\begin{theorem}[Nerve Theorem] 
\label{thm:Nerve}
If  $\UU$ is a numerable good cover of $X$, then the Aleksandroff map $\varphi: X \to |N(\UU)|$ is a homotopy equivalence.
\end{theorem}

As a consequence, a paracompact space admits a finite good cover if, and only if it is 
homotopy equivalent 
to a finite (simplicial or CW) complex. In the literature one can find many variants of the 
Nerve theorem, which under different 
sets of assumptions show that the Alexandroff map is a homotopy equivalence, or a weak
homotopy equivalence, or a homology equivalence, etc.

The idea of covering type provides an important link between good covers and minimal triangulations.
In general, 
given a polyhedron $P$, one often looks for triangulations of $P$ with the minimal number of vertices.
Again, there are many variants and aspects of the problem so we 
introduce the following systematic notation: given a compact polyhedron $P$ we denote
by $\Delta(P)$ the minimal number of vertices in a triangulation of $P$, i.e.,
$$ \Delta(P): = \min\big\{\mathrm{card}(K^{(0)})\big|\, |K|\approx P\big\}.$$
If $P$ is a manifold then one is principally  interested in combinatorial (or PL-) 
triangulations, i.e., triangulations in which the links of vertices are combinatorial spheres. 
Thus for every PL-manifold $M$ we define
$$\Delta^{PL}(M):= \min\, \big\{\mathrm{card}(K^{(0)})\big|\, K \,{\text{is a}} \; 
{\text{PL-triangulation of }} \; M \big\}\,.$$
Computing $\Delta$ or $\Delta^{PL}$ is a hard and intensively studied problem of combinatorial
 topology - see Datta \cite{Datta}
and Lutz \cite{Lutz} for surveys of the vast body of work related to this question.

Every triangulation of a space $X$ gives rise to a  good cover of $X$ by taking open stars of 
vertices in the triangulation. 
On the other hand, if $\UU$ is a good cover of $X$, then $|N(\UU)|$ is homotopy equivalent to $X$ 
by the Nerve Theorem. Therefore, whenever $X$ has the homotopy type of compact polyhedron we may 
introduce a homotopy analogue of $\Delta(P)$ as
$$\Delta^{\simeq}(X): = \min\{\Delta (P)\mid P \simeq X \}\,.$$
Clearly, $\Delta^{\simeq}(K)$ is a lower bound for other invariants, since 
$$\Delta^{\simeq}(P)\le \Delta(P),$$ 
and if $M$ is a PL-manifold
$$\Delta^{\simeq}(M)\le \Delta^{PL}(M).$$ 

On the other hand $\Delta^\simeq(X)$ is directly related to the covering type.

\begin{theorem}
\label{thm:min tgl}
If $X$ has the homotopy type of a finite polyhedron, then
$$  \ct(X)=\Delta^\simeq (X)$$
\end{theorem}
\begin{proof} 
Let $\mathcal{U}$ be a good cover of $X$ of cardinality $\ct(X)$.
The nerve $N(\mathcal{U})$ has $\ct(X)$ vertices, and  $|N(\UU)|\simeq X$ by the Nerve theorem,
which implies $\Delta^\simeq(X)\,\leq \,  \ct(X)\,.$

Conversely, if $K$ is a simplicial complex  such that $|K|\simeq X$, 
then the cover of $|K|$ by open stars of vertices is a good cover of $|K|$, therefore
$ \Delta^\simeq(X)\, \geq \,\ct(X) \,.$
\end{proof}

As a consequence, it is of great interest to find good lower estimates for 
$\ct(X)$ as they in turn give lower bounds for the size of minimal triangulations. 
On the other hand, it is easy to find examples where $\ct(P)$ is strictly smaller than 
$\Delta(P)$ (cf. \cite[Example 1.3]{K-W}).
The upper estimates for $\ct(X)$ are therefore less relevant as a tool for the study of
minimal triangulations. Indeed, upper estimates 
of $\ct(X)$ are usually obtained by finding explicit triangulations of $X$, while the lower estimates
are based on certain obstructions. The latter is a natural setting for the methods of homotopy theory,
and is one of the reasons why the relation with $\ct(X)$  is so useful. 

It appears to be much harder to distinguish between $\ct(M)$ and $\Delta^{PL}(M)$ when $M$ 
is a closed manifold. 
Borghini and Minian \cite{B-M} have recently proved that with one exception 
$\ct(M)=\Delta^{PL}(M)$ for all closed surfaces (both orientable and non-orientable). The only
exception is the genus two orientable surface $T\# T$, as they prove 
that $\ct(T\# T)=9$ while $\Delta^{PL}(T\# T)=10$, 
which means that the minimal triangulation of $T\# T$ has 10 vertices but there exists 
a 2-dimensional complex on 9 vertices whose geometric realization is homotopy equivalent 
to $T\# T$. Since the covering type is a homotopy invariant, it would be of interest to find 
conditions under which the covering type of a closed manifold $M$ coincides with the minimal number 
of vertices in a triangulation of $M$.

The paper is organized as follows. In the next two sections we relate the Lusternik-Schnirelmann category and the cohomology ring
of a space to the covering type and derive a series of lower estimates for the covering type. In the last section we study the effect that 
suspensions and wedge-sums have on the covering type, and give some useful upper and lower estimates for the covering type of a Moore space.

%%----------------------------------------------------------------------------------------------------------------------------------------------------
%% LS-CATEGORY ESTIMATES
%%----------------------------------------------------------------------------------------------------------------------------------------------------

\section{LS-category estimates}

Recall the definition of the Lusternik-Schnirelmann (LS-)category of a space $X$. A subset $A\subseteq X$ is \emph{categorical} if the inclusion 
$A\hookrightarrow X$ is homotopic to the constant map. Then \emph{LS-category} of $X$, denoted $\cat(X)$, is the minimal $n$, for which $X$ can be covered 
by $n$ open categorical subsets. A standard reference is \cite{CLOT}

\begin{remark}
Contractible subsets of $X$ are clearly categorical, but the converse is not true - e.g., the sphere
is a categorical subset of the ball. There is a related 
concept called \emph{geometric category}, defined as  the minimal cardinality of a cover of $X$ 
by open contractible sets (see \cite[Chapter 3]{CLOT}). 
Like the strict covering type, the geometric category is not a homotopy invariant of $X$, so 
one defines the \emph{strong category}, $\mathrm{Cat}(X)$,
as the minimum of geometric categories of spaces that are homotopy equivalent to $X$. Although 
the categorical sets may be very different from contractible
ones, the following remarkable relation holds: $\cat(X)\le\mathrm{Cat}(X)\le\cat(X)+1$ 
(see \cite[Proposition 3.15]{CLOT}). 

Little is known about analogous relationships between the covering type and the strict covering type.
For the wedge $W_n$ on $n$ circles  Karoubi and Weibel \cite[Proposition 4.1]{K-W} show that 
$\sct(W_n)=n+2$, while $\ct(W_n)=\left\lceil \frac{3+\sqrt{1+8n}}{2} \right\rceil$, which means that
the difference between the two can be arbitrarily large. Furthermore, as we mentioned before, 
Borghini and Minian \cite[Proposition 3.4]{B-M} found an example of a closed surface for which the 
covering type is one less than its strict covering type.
\end{remark}

The relation between the category and the covering type of a space is also complicated. 
For spheres $\cat(S^n)=2$ while $\ct(S^n)=n+2$. 
Neither of them determines the other. We will give below examples of spaces that have the same
covering type and yet the difference between
respective categories is as big as we want. Nevertheless, if the category of a space is $n>1$, then
its (homotopy) dimension is at least $n-1$ and so its covering
type is at least $n+1$ (because it is not contractible). Roughly speaking, spaces with big category
cannot have small covering type. We are going to make this statement more precise in the rest of this
section. 

Let us begin with a list of facts on which we will base the proofs of our results.
\begin{itemize}
\item[A)]By the Nerve theorem, if $X$ admits a good cover $\mathcal{U}$ of order $\leq n$ (i.e., 
at most $n$ different sets have non-empty intersection), then $X$ is
homotopy equivalent to a simplicial complex of dimension $n-1$.
\item[B)]By the Nerve theorem, if $U_1, \dots\,, U_n$ are elements of a good cover that intersect 
non-trivially, then $U_1\cup\,\cdots\, \, \cup U_n$ is 
homotopy equivalent to $\Delta_{n-1}$, and therefore contractible.
\item[C)] If $\cat(X) \geq n$, $X=U\cup V$, where $U, \,V$ open and $U$ is contractible 
(or more generally, $U$ is categorical in $X$), then
$\cat(V)\geq n-1$. This is obvious, because $\cat(V)<n-1$ would imply $\cat(U\cup V)<n$. 
\item[D)] $\cat(X)\leq \hdim(X) +1$, where $\hdim(X)$ is the \emph{homotopy dimension} of $X$, 
defined as 
$\hdim(X):=\min \big\{ \dim (Y) \mid Y\simeq X,\ Y\ \text{CW-complex}\big\}.$
The claim follows from the classical estimate $ \cat (Y)\leq \dim Y+1$ and the homotopy
invariance of LS-category.
\end{itemize}

\begin{theorem}
\label{thm:cat}
$$ \ct (X) \,\geq\, \frac{1}{2} \, \cat(X) \, (\cat(X) +1) $$
\end{theorem}
\begin{proof} Assume that $X$ has a good cover  $\mathcal{U}$ of
cardinality $\ct(X)$. We proceed by induction. If $\cat(X) =1 $ then
$X$ is contractible, thus $\ct(X)=1$ and the inequality reduces to
$1\geq 1$.

Assume that the estimate holds for spaces with category $\leq n$, and let $\cat(X) = n+1$. 
Then D) implies $\hdim (X) \geq n$, so by A) there exist sets $U_1,\ldots, U_{n+1} \in \mathcal{U} $ which intersect non-trivially. 
Let $U:= U_1\cup\ldots \cup\, U_{n+1}$ and let $V$ be the union of remaining  elements of $\mathcal{U}$. Then $U$ is contractible by B), which gives $\cat(V)\geq n$ by C). 
We use the induction assumption to compute 
$$\ct(X) \geq (n+1)+\frac{1}{2}\, n(n+1) = \frac{1}{2} \,
(n+1)(n+2)\, .$$
\end{proof}

Direct application of the theorem gives the following estimates. For spheres $\cat(S^n)=2$, therefore 
$\ct(S^n)\ge 3$, and for surfaces (with the exception of $S^2$) $\cat(P)=3$, therefore  
$\ct(P)\ge 6$. Furthermore, for real and complex projective spaces 
$\cat(\RR P^n )=\cat(\CC P^n )=n+1$, so that 
$\ct(\RR P^n) \geq \frac{(n+1)(n+2)}{2}$ and $\ct(\CC P^n) \geq \frac{(n+1)(n+2)}{2}$.

A comparison with the results of \cite{K-W} shows that some of the above estimates are not optimal and can be improved. In fact, we neglected the information about 
the dimension and connectivity of $X$, which also have an impact on the covering type. By taking these data into account we obtain much better estimates (except for real projective spaces, 
which are only $0$-connected and the category is directly related to the dimension). Nevertheless it is interesting to observe that the covering type increases (at least) 
quadratically with the category of the space.

\begin{theorem}
\label{thm:cat+dim} 
$$\ct(X) \geq 1 +\hdim(X) + \frac{1}{2}\,\cat(X)(\cat(X)-1)$$
\end{theorem}
\begin{proof}
By fact A) there exist sets $U_1,\ldots,U_{\hdim(X)+1} \in \mathcal{U}$ that intersect non-trivially. 
Let $U= U_1\cup\ldots\cup\, U_{\hdim(X)+1}$, and let
$V$ be the union of remaining elements of $\mathcal{U}$. As above $ \cat(V) \geq \cat(X)-1$, which together with Theorem \ref{thm:cat} yields
$$ \ct(X)\geq (\hdim(X)+1) + \frac{1}{2}\,\cat(X) (\cat(X) -1).$$
\end{proof}

A similar approach can be used to  estimate  the minimal number of points (vertices) that are required in order to triangulate 
a given PL-manifold. 
Recall that a triangulation of a manifold is \emph{combinatorial} if the links of all vertices are triangulated spheres. 
Then we have the following

\begin{corollary}
\label{cor:cat+conn}
Let $M$ be a $d$-dimensional and $c$-connected closed PL-mani\-fold.
Then 
$$\Delta^{PL}(M)\ge 1+d+ c \cdot (\cat(M)-2) + \frac{1}{2}\,\cat(M)(\cat(M)-1).$$ 
\end{corollary}
\begin{proof}
Note that the case $c=0$ is covered by Theorem \ref{thm:cat+dim}.
Similarly, the claim is correct if we assume $c=d-1$, because that implies $M\simeq S^d$ and 
$\ct(S^d)=d+2$.
Finally, if $d=2$ and $M$ is 0-connected but not 1-connected, then $M$ is a closed surface other 
than sphere, hence $\cat(M)=3$, and the above formula claims  that $\Delta^{PL}(M)\ge 6$, which we 
already know. Therefore we may assume from this point on that  $d\ge 3$ and $1\le c\le d-2$.

Let $K$ be a combinatorial triangulation of $M$ and let $\UU$ be the good cover of $M$ formed
by  open stars of vertices of $K$. Since $M$ is $d$-dimensional,
there are at least $d+1$ open stars that intersect non-trivially. Their union $U$ is the simplicial
neighbourhood of 
a $d$-dimensional simplex, and the intersection of $U$ with the union $V$ of all other open stars is homeomorphic to $S^{d-1} \times (0,1)$. 

Since $d\ge 3$, the Seifert-van Kampen theorem implies that $V$ is simply connected. Furthermore, since $c\le d-2$ the Mayer-Vietoris 
sequence for the cover $\{U,V\}$ of $M$ shows that $H_i(V)\cong H_i(M)$ for $i\le c$, thus $V$ is also $c$-connected.  

We have seen previously that $\cat(V)\ge\cat(M)-1$, so we use the known inequality (see \cite{CLOT}) 
\begin{equation}\label{category by connectivity} \cat(V) \leq \frac{\hdim (V)}{c+1} +1\end{equation}
to deduce that $\hdim(V)\geq (c+1) (\cat(M) -2)$. As before, this implies that at least $(c+1)(\cat(M)-2)+1$ of the open stars
that cover $V$ have a non-trivial intersection, and that $W$, the union of the remaining open stars has category $\cat(W)\geq \cat(M)- 2$. 
By applying Theorems \ref{thm:min tgl} and \ref{thm:cat+dim} to estimate the number of vertices 
in $W$ we may conclude that $K$ has at least
$$ 1+d+(c+1)(\cat(M)-2) +1 +\frac{1}{2}\,(\cat(M)-1)(\cat(M)-2)=$$
$$= 1+d +c\cdot (\cat(M)-2) + \frac{1}{2}\,\cat(M)(\cat(M)-1)$$
vertices.
\end{proof}

Observe that this estimate is a strict improvement of Theorem \ref{thm:cat+dim} for all PL-manifolds which are at least $1$-connected
and are not spheres. For example, it shows that every triangulation of $\CC P^n$ requires at least $\frac{1}{2}\,n(n+7)$ vertices. 

%\begin{remark} By substituting the inequlity (\ref{category by connectivity}), applied to 
%${\rm cat} (M)$, into the statement of Corollary \ref{cor:cat+conn} we get a lower bound for
%the number of vertices in a triangulation of a PL-manifold as 
%$$ \frac{d(d-c-1)}{2(c+1)^2} +2d -c +1 \,.$$
%This estimate is quadratic in $d$, so for a fixed $c$, and  $d$ sufficiently large, it is better than 
%a similar result of \cite{Bre-Kuh} which states that this number is greater or equal to 
%$ 2d -c +3$ with the same assumptions on $M$.
%\end{remark}

We may also reverse  the above estimates to obtain upper bounds for the category of a space based on the cardinality of good cover or the number of vertices in a triangulation.

\begin{corollary}
\label{cor:cat by ct}
Assume that $X$ admits a good cover with $n$ elements. Then the category of $X$ is bounded above by
$$ \cat(X) \leq \frac{-1+\sqrt{1+8n}}{2}.$$
If the dimension of $X$ is known we have also a better estimate
$$ \cat(X) \leq \frac{1+\sqrt{1+8(n-\hdim(X)-1)}}{2}.$$
\end{corollary}
\begin{proof} The estimates are easily proved by solving the inequalities in Theorem \ref{thm:cat} 
and Theorem \ref{thm:cat+dim} for $\cat(X)$.
\end{proof}

We mentioned the general estimate $\ct(X)-\hdim(X)\ge 2$ that holds for every non-contractible 
space $X$. By using the above inequalities we easily see that $\ct(X)-\hdim(X)\le 3$ implies 
$ \cat(X) \le 2$ and that $\ct(X)-\hdim(X)\le 6$ implies $ \cat(X) \le 3$. This leads to 
the following interesting result. 

\begin{corollary}
Assume that $\ct(X)-\hdim(X)\le 3$ or that $X$ is a closed manifold (of dimension at least 3) 
and $\ct(X)-\dim(X)\le 6$. Then the fundamental group of $X$ is free. 
\end{corollary}
\begin{proof}
If $\ct(X)-\hdim(X)\le 3$ then by \ref{cor:cat by ct} the category of $X$ is at most 2,
and a well-known result see \cite[Section 1.6]{CLOT} implies that $\pi_1(X)$ is a free group. 

Similarly, if $\ct(X)-\hdim(X)\le 6$, then the category of $X$ is at most 3. 
Our claim then follows from the main result of \cite{DKR} that the fundamental group of 
a closed manifold whose fundamental group is not free must have category at least 4. 
Note that the second statement holds for closed surfaces as well, with the exception of the torus, 
the projective plane and the Klein bottle. 
\end{proof}

%%----------------------------------------------------------------------------------------------------------------------------------------------------
%% COHOMOLOGICAL ESTIMATES
%%----------------------------------------------------------------------------------------------------------------------------------------------------

\section{Cohomological estimates} 

It is well-known that the Lusternik-Schnirelmann category of a space $X$ is closely related to the structure of the cohomology ring $\tilde{H}^*(X)$. 
Indeed, $\cat(X)$ is bounded below by the so-called \emph{cup-length} of $X$, which is defined as the maximal number of factors among all non-trivial products 
in $\tilde{H}^*(X)$ (and with any coefficients, see \cite[Proposition 1.5]{CLOT}). However, that estimate does not involve the respective dimensions of the factors 
in the product. We are going to show that the latter play an important role in the estimate of covering type, which will lead to considerable
improvements in  our estimates of the covering type of $X$ .

Given an $n$-tuple of positive integers $i_1,\ldots,i_n \in \NN$ we will say that a space $X$ \emph{admits 
an essential $(i_1,\ldots,i_n)$-product} if there are cohomology classes $x_k \in H^{i_k}(X)$, such that 
the product $x_1\cdot x_2\cdot\ldots\cdot x_n$ is non-trivial. 
For every $(i_1,\ldots,i_n)$ there exist a space $X$ that admits an essential $(i_1,\ldots,i_n)$-product, 
for example we can take $X=S^{i_1}\times\cdots\times S^{i_n}$.
Clearly, if $X$ admits an essential $(i_1,\ldots,i_n)$-product then so does every  $Y\simeq X$, 
since their cohomology rings are isomorphic. We may therefore define the \emph{covering type of the $n$-tuple of positive integers} $(i_1,\ldots,i_n)$ as
$$\ct(i_1,\ldots,i_n):=\min\big\{\ct(X)\mid X\,\text{admits\ an\ essential}\ (i_1,\ldots,i_n)\mathrm{-product}\big\}$$ 
The following proposition follows immediately from the definition.
\begin{proposition}
\label{prop:coho est}
$$ \ct (X) \geq \max \{\ct(|x_1|,\ldots,|x_n|)\mid \  \mathrm{for\  all}\ 0\neq x_1\cdots x_n\in H^*(X)\}$$
\end{proposition}

Although the covering type of a specific product of cohomology classes may appear as a coarse estimate it will serve very well our purposes. 
We will base our computations on the following technical lemmas. The first is a standard argument that we give here for the convenience of the reader. 

\begin{lemma}
Let $X=U\cup V$ where $U,V$ are open in $X$, and let $x,y\in\widetilde H^*(X)$ be cohomology classes whose product $x\cdot y$ is non-trivial.
If $U$ is categorical in $X$ then $i_V^*(x)$ is a non-trivial element of $H^*(V)$ (here $i_V$ stands for the inclusion map $i_V: V \hookrightarrow X$).
\end{lemma}
\begin{proof}
Assume by contradiction that $i_V^*(x)=0$. Exactness of the cohomology sequence 
$$H^*(X,V) \stackrel{j_V^*}{\longrightarrow} H^*(X) \stackrel{i_V^*}{\longrightarrow} H^*(V)$$
implies that there is a class $\bar x\in H^*(X,V)$ such that $j_V^*(\bar x)=x$. 
Moreover $i_U^*(y)=0$, because $i_U\colon U\hookrightarrow X$ is null-homotopic, so there 
is a class $\bar y\in H^*(X,U)$ such that $j_U^*(\bar y)=y$. 
Then $x\cdot y=j_V^*(\bar x)\cdot j_U^*(\bar y)$ is by naturality equal to the image of $\bar x\cdot\bar y\in H^*(X,U\cup V)=0$, therefore $x\cdot y=0$, which 
contradicts the assumptions of the lemma. 
\end{proof}

By inductive application of the above lemma we obtain the following:

\begin{lemma}
\label{lem:subproducts}
Let  $x_1,\ldots,x_n\in\widetilde H^*(X)$ be cohomology classes whose product $ x_1\cdots x_n$ is non-trivial, and let $X=U_1\cup\ldots\cup U_k\cup V$ where $U_1,\ldots,U_k$ 
are open, categorical subsets of $X$, and $V$ is open in $X$. Then the product of any $(n-k)$ different classes among $i^*_V(x_1),\ldots,i^*_V(x_n)$ is a non-trivial class in  $H^*(V)$.
\end{lemma}

The following simple extension of \cite[Proposition 3.1]{K-W} will allow a slightly better estimate of $\ct(i_1,\ldots\,i_n)$ if the parameters $i_1,\ldots\,i_n$ are not all equal. 

\begin{lemma}
\label{lem:two homologies}
If $X$ has non-trivial reduced homology groups in different dimensions, then $\ct(X) \geq \hdim(X)+3$.
\end{lemma}
\begin{proof}
If $\ct(X) \leq \hdim(X)+2$, then $X$ is homotopy equivalent to a subcomplex of $\Delta_{\hdim(X)+1}$. The only subcomplex of $\Delta_{\hdim(X)+1}$ that has
homotopy dimension equal to $\hdim(X)$ is $\partial \Delta_{\hdim(X)+1}$, which has only one non-trivial reduced homology group. 
\end{proof}

We are ready to prove the main result of this section, an 'arithmetic' estimate for the covering type of a $n$-tuple:

\begin{theorem}
\label{thm:arithmetic}
$$ \ct(i_1,\ldots\,i_n)\geq  i_1 + 2\, i_2 + \, \cdots\, + n i_n + (n+1) $$
If $i_1,\ldots\,i_n$ are not all equal, then 
$$ \ct(i_1,\ldots\,i_n)\geq  i_1 + 2\, i_2 + \, \cdots\, + n i_n + (n+2) $$
\end{theorem}
\begin{proof} 
The first statement can be proved by induction. Unfortunately, the same approach is not sufficient to prove the stronger statement, 
and a modified inductive argument turns out to be quite complicated, and we find it easier to give a direct proof. 
Although the second proof covers the first statement as well, we believe that it still of some interest to 
be able compare the two methods. 

Toward the proof of the first statement, we begin the induction by observing that if $0\neq x_1\in H^{i_1}(X)$ then $\hdim(X) \geq i_1$, 
hence $\ct(i_1)\ge i_1+2$ by \cite[Proposition 3.1]{K-W}.

Assume that the estimate holds for all sequences of $(n-1)$ positive integers and consider the classes $x_1\in\widetilde H^{i_1}(X),\ldots,x_n\in\widetilde H^{i_n}(X)$ 
such that the product $x_1 \cdots x_n \in H^{i_1+\ldots +i_n}(X)$ is non-trivial. 
The cohomological dimension of $X$ is at least $i_1+\ldots+i_n$, therefore in every good cover $\mathcal{U}$  of $X$ one can find $i_1+\cdots+i_n+1$ elements
that intersect non-trivially. Denote their union by $U$ and the union of  the remaining elements of $\mathcal{U}$ by $V$. Then $U$ is contractible and by 
Lemma \ref{lem:subproducts}  there exists in  $\widetilde H^*(V)$ a non-zero  product  of elements whose degrees are $i_2,\dots,i_n$. By induction we obtain
$$ \ct(X) \geq (i_1+\cdots+i_n +1)+ (i_2 +2 i_3+\ldots+(n-1) i_n+n) =$$
$$= i_1+2 i_2+\ldots+n i_n + (n+1), $$
which proves the first statement.

For the second statement, let $\UU$ be a good cover of $X$, and assume that the product of classes $x_1\in\widetilde H^{i_1}(X),\ldots,x_n\in\widetilde H^{i_n}(X)$ is non-trivial. 
As before, there exists $\UU_1\subseteq \UU$, such that $\UU_1$ contains $(i_1+\ldots+i_n+1)$ sets that intersect non-trivially. If we denote by $V_1$ the union of sets in $\UU-\UU_1$, 
then by Lemma \ref{lem:subproducts} the restriction to $V_1$ of any sub-product of $x_1 \cdots x_n $  of length $(n-1)$ is non-trivial. In particular, $H^{i_2+\ldots+i_n}(V_1)\neq 0$, 
and so there exists $\UU_2\subseteq \UU-\UU_1$, such that $\UU_2$ contains $(i_2+\ldots+i_n+1)$ sets that intersect non-trivially. 
By continuing this procedure we end up with disjoint collections $\UU_1,\ldots,\UU_{n-1}\subseteq \UU$, where each $\UU_k$ has $(i_k+\ldots+i_n+1)$ elements and the union of its elements 
is contractible. 

Let $V$ denote the union of all elements in $\UU_n:=\UU-\UU_1-\ldots-\UU_{n-1}$. By Lemma \ref{lem:subproducts} $H^*(V)$ has non trivial cohomology classes in dimensions $i_1,\ldots,i_n$. 
Since we assumed that they are not all equal, Lemma \ref{lem:two homologies} implies that $\UU_n$ has at least $i_n+3$ elements. By adding up the cardinalities 
of all $\UU_k$ we conclude that $\UU$ has at least $i_1 + 2\, i_2 + \, \cdots\, + n i_n + (n+2)$ elements.
\end{proof}

It is worth to emphasize that it is usually not difficult to identify the cup product in $\widetilde H^*(X)$ which provides the best estimate for the covering type.
In particular, it  clearly makes sense to  consider only products whose terms have non-decreasing degrees. The rest of the section is dedicated to 
computations of specific examples (projective spaces, products of spheres, etc.) based on Theorem \ref{thm:arithmetic}.

\begin{corollary}\label{cor:ct projective spaces}
The covering type of projective spaces is bounded by:\\ $\ct(\RR P^n)\ge\frac{1}{2}(n+1)(n+2)$, $\ct(\CC P^n)\ge (n+1)^2$, $\ct(\HH P^n)\ge (n+1)(2n+1)$.
\end{corollary}
\begin{proof}
By Theorem \ref{thm:arithmetic} it is sufficient to estimate
$$\ct(\underset{n}{\underbrace{k,\ldots,k}})\ge  k + 2 k+\ldots+ nk + n  +1 =\frac{1}{2} (n+1)(kn+2),$$
because the real, complex and quaternionic projective spaces correspond respectively to cases $k=1,2,4$.
\end{proof}

For a product of spheres $X=S^{i_1}\times \cdots\times S^{i_n}$ where $i_1\le\ldots\le i_n$ are not all equal Theorem \ref{thm:arithmetic}
yields $\ct(X)\ge i_1 + 2\, i_2 + \, \cdots\, + n i_n + (n+2)$, while for a product of spheres of same dimension we get 
$$\ct((S^i)^n)\ge \frac{(n+1)(ni+2)}{2}\,.$$ The last estimate can be sometimes improved by ad-hoc methods - see Example \ref{examples by linearly independent}.

\begin{corollary}
\label{cor:U(n)}
The covering type of unitary groups is estimated as 
$$\ct(U(n)) \geq  \frac{1}{6}(4n^3+3n^2+5n+12) \ \ \text{and} \ \ \ \ct (SU(n)) \geq  \frac{1}{6}(4n^3-3n^2+5n+6).$$
\end{corollary}
\begin{proof}
The cohomology algebra $H^*(U(n))$ is the exterior algebra on generators in dimensions $1,3,\ldots,(2n-1)$, while 
$H^*(SU(n))$ is the exterior algebra on generators in dimensions $3,5, \ldots,(2n-1)$. Theorem \ref{thm:arithmetic} gives
$$\ct(U(n)) \geq \ct(1, 3,\ldots, 2n-1)\ge 1+ 2\cdot 3+3\cdot 5+\ldots+ n\cdot (2n-1) +(n+2)=$$
$$= \frac{1}{6}(4n^3+3n^2+5n+12)$$
and
$$\ct(SU(n)) \geq \ct(3,5,\ldots, 2n-1) \ge 1\cdot 3+ 2\cdot 5+\ldots+ (n-1)\cdot (2n+1) +(n+1)=$$
$$= \frac{1}{6}(4n^3-3n^2+5n+6).$$
\end{proof}
  
The LS-category of unitary groups is $\cat(U(n))=n$ and $\cat(SU(n))=n-1$ 
(see \cite[Theorem 9.47]{CLOT}), so our of the covering type estimate is 
a cubical function of the category (as compared with results from Section 2 where we 
obtained a general quadratic relation between the category and the covering type). 
\begin{remark}
The estimates of Corollary \ref{cor:ct projective spaces}  applied to the number of vertices 
of triangulation of $\RR P^n$ and $\CC P^n$ or spaces with the same cohomology algebra  reproves 
the result of  \cite{Ar-Ma}. The corresponding estimate for $\HH P^n$ was not stated  in the
literature, up to our knowledge.

The estimate of  number  of  vertices in a triangulation of $U(n)$,  or $SU(n)$,  that follow 
from   our estimate of the covering type in Corollary \ref{cor:U(n)} is   new.
\end{remark}

The computation for unitary groups can be easily extended to finite, (homotopy) associative 
$H$-spaces, i.e., spaces with a continuous product that is associative only 
up to a suitable homotopy (see \cite[Section III,4]{Whitehead}). 
In fact the $\ZZ_p$-cohomology of 
a finite associative $H$-space is given as (see \cite[Theorem III,8.7]{Whitehead}).
$$ H^*(X;\ZZ_p) \cong {\overset{n}{\underset{i=1}\otimes}}\bz_p[x_i]/(x_i^{k_i})$$
where $k_i$ is a power of 2 if $p=2$, while for $p$ odd there are two cases: $k_i=2$ if $|x_i|$ is odd, and $k_i$ is a power of $p$ if $|x_i|$ is even. Thus, for a given prime $p$ we may 
consider the corresponding structure of $ H^*(X;\ZZ_p)$ and define 
$$\ct_p(X):=\ct\big(\ \underset{k_1-1}{\underbrace{|x_1|, \,
\dots\,,|x_1|}}, \, \dots\, ,\underset{k_n-1}{\underbrace{|x_n|,
\, \dots\,,|x_n|}}\ \big) $$
A specific $H$-space may have a trivial cohomology structure with respect to some coefficient fields
and very rich with respect to other fields, yielding different values of $\ct_p(X)$. 
The crucial observation is that the covering type yields 
a uniform bound for all of them so we have the following inequality
$$ \ct(X) \geq \max_{p\ \text{prime}}\{\ct_p(X)\}$$ 
In particular, we can easily compute lower estimates for the covering type of all classical Lie
groups, since their cohomology rings are well-known.

Let us mention that for spaces whose cohomology algebra has several linearly independent generators 
in low dimensions it is possible to improve the general estimates 
of the covering type. Since the actual improvements arise only in few cases we do not attempt to develop a theory but instead illustrate this method on an example.

\begin{example}\label{examples by linearly independent}\rm
We are going to estimate $\ct(S^1 \times S^1\times S^1)$. 
Let $\mathcal{U}$ be a good cover of some $X$ that is homotopy
equivalent to $S^1\times S^1\times S^1$. Then with respect to any field coefficients we have $H^*(X)
\cong \Lambda (x,y,z)$, where $|x|=|y|=|z|=1$. 
Since $\hdim(X)=3$ there are at least $4$  open sets, say $U_1,\, U_2,\, U_3,
U_4\in\UU$ that intersect non-trivially.
The union of the remaining  elements of $\mathcal{U}$ has category at least $3$ and $\hdim$ at least $2$. By Theorem \ref{thm:cat+dim} a good cover 
of it has at least six sets, so there are also sets $U_5,\ldots,U_{10}\in\UU$. 
Observe that by Nerve theorem $H^1(U_5\cup U_6\cup U_7)$ is at most 1-dimensional, and
similarly for $H^1(U_8\cup U_9\cup U_{10})$.
Since $H^1(X)$ is $3$-dimensional, the kernel of the restriction homomorphism
$$ H^1(X) \longrightarrow H^1(U_5\cup U_6\cup U_7) \oplus H^1(U_8\cup U_9\cup U_{10})$$ 
contains a non-trivial element $u\in H^1(X)$. Moreover, the kernel of the restriction homomorphism
$$ H^1(X)\longrightarrow H^1(U_5\cup U_6 \cup U_7)$$ 
is at least $2$-dimensional so we may find in it a non-trivial element $v\in H^1(X)$ which is linearly independent from $u$. Finally, we
can choose $w\in H^1(X)$ such that the set $\{u,v,w\}$ is a basis of $H^1(X)$.

The restriction of $u$ to $U_8\cup U_9\cup U_{10}$ is trivial, so by exactness of 
the cohomology sequence of the pair there exists 
$$\bar{u}\in H^1(X, U_8\cup U_9\cup U_{10})$$ 
such that $ j^*(\bar{u}) =u$ (where $j^*$ denotes the homomorphism induced by the inclusion of $X$ 
in the pair). Similarly, one can find  
$$\bar{v} \in H^1(X, U_5\cup U_6\cup U_7), \;\; \bar{w}\in H^1(X, U_1\cup \, \dots \, \cup
U_{4})$$ satisfying  $ j^*(\bar{v}) =v$ and $ j^*(\bar{w}) =w$. Since $0\neq u\cdot v\cdot w =
j^*(\bar{u}\cdot\bar{v}\cdot \bar{w})$ and 
$\bar{u}\cdot\bar{v}\cdot \bar{w}\in H^3(X,U_1\cup\ldots\cup U_{10})$, 
we conclude that $X \neq U_1\cup U_2\, \cdots\, \cup U_{10}$, therefore $\ct(X) \geq 11$.

Note that it is an improvement of estimate of Theorem \ref{thm:arithmetic}  which gives 
$$ \ct(X) \geq \ct(1,1,1)\ge  1+2+3+4= 10$$
It is worth pointing out that for the product of four circles the two methods give the same 
estimate 
$\ct(S^1 \times S^1\times S^1\times S^1) \ge 15$, while for more than $4$ generators Theorem \ref{thm:arithmetic} 
yields a better estimate of the covering type. 
\end{example}

%%----------------------------------------------------------------------------------------------------------------------------------------------------
%% MOORE SPACES
%%----------------------------------------------------------------------------------------------------------------------------------------------------

\section{Moore spaces}

In this section we estimate the covering type of various Moore spaces and use the results to derive estimates for related spaces. 
Recall that for every abelian group $A$ and positive integer $i$ one can construct a CW complex $X$ with 
$$\widetilde H_k(X)=\left\{\begin{array}{ll}
A & k=i\\
0 & k\ne i
\end{array}\right.
$$
Any such space is called a \emph{Moore space} of type $M(A,i)$ (cf. \cite[Example 2.40]{Hatcher}). 
If $i>1$ we normally assume that $X$ is simply connected, because then $A$ and $i$ uniquely 
determine the homotopy type of $X$ (see \cite[Example 3.34]{Hatcher}), and we may write $M(A,i)$
instead of $X$. Homotopy uniqueness fails if $i=1$ so we normally consider specific constructions
of $M(A,1)$ (see bellow). The simplest examples are wedges of spheres: 
$r$-fold wedge of $i$-dimensional spheres is a Moore space of type $M(\ZZ^r,i)$. 

By Theorem \ref{thm:min tgl} every space $X$ with $\ct(X)=n$ is homotopy equivalent to a 
subcomplex of $\Delta_{n-1}$. Therefore, for any given $n$ 
there exist only finitely many homotopy types of spaces whose covering type is equal to $n$, and we may even attempt a classification, 
at least for small values of $n$. For each $n$ there is always the trivial example of a space with $\ct(X)=n$, namely the discrete space with 
$n$ points. These are the only spaces with covering type 1 or 2. The first non-trivial example is the circle, whose covering type is 3, and 
belongs to the family of spheres $S^n$ whose covering is $\ct(S^n)=n+2$. Apart from the discrete space and the sphere, there are two other
spaces with covering type 4, namely the wedges of 2 and of 3 circles. Similarly, the spaces of covering type 5 are wedges of spheres of various 
dimensions. The number of homotopy distinct complexes increases rapidly with the covering type, but there is a reasonably complete classification
for manifolds whose covering type is at most 11 (cf. \cite[Section 5]{Datta}).

The following theorem gives the covering type of Moore spaces with free homology.

\begin{theorem}
\label{thm:ct wedge}
$\ct\big(M(\ZZ^r,i)\big)=n$, where $n$ is the minimal integer for which ${{n-1}\choose {i+1}} \ge r.$
\end{theorem}
\begin{proof} 
The case $i=1$ is covered by \cite[Proposition 4.1]{K-W}, so we may assume $i>1$, and thus avoid complications with the fundamental group.

Let us first compute the homology of $\Delta_{n-1}^{(i)}$, the $i$-th skeleton of the $(n-1)$-dimensional simplex. The simplicial chain complex 
of $\Delta_{n-1}^{(i)}$ is obtained by truncating the simplicial 
chain complex for $\Delta_{n-1}$ at degree $i$: 
$$ C_i(\Delta_{n-1})\stackrel{\partial_i}{\longrightarrow} C_{i-1}(\Delta_{n-1}) \longrightarrow \cdots\cdots 
\longrightarrow C_0(\Delta_{n-1}) \stackrel{\partial_0}{\longrightarrow} C_{-1}=\ZZ $$
The homology of $\Delta_{n-1}$ is trivial, so the above chain complex is exact, except at the beginning. 
The rank of each $C_k(\Delta_{n-1})$ is ${n \choose {k+1}}$, and the rank of $H_i(\Delta_{n-1}^{(i)})=\ker\partial_i$ can be computed by 
exploiting the exactness: 
$$\rank(\ker\partial_i)={n \choose {i+1}}-{n \choose i}+\ldots+(-1)^i{n \choose 1}+(-1)^{i+1}
{n \choose 0}={n-1 \choose i+1}.$$
Since $\Delta_{n-1}^{(i)}$ is clearly  simply connected by the standard computation of 
the fundamental group of a simplicial complex, we conclude that it is a Moore space of type
 $M(\ZZ^{n-1 \choose i+1},i)$.

It is obvious from the definition of simplicial homology that the rank of $H_i(\Delta_{n-1}^{(i)})$ 
is maximal among all sub-complexes of $\Delta_{n-1}$. 
Therefore, if $r>{{n-1}\choose {i+1}}$, then $M(\ZZ^r,i)$ cannot be represented by a subcomplex of 
$\Delta_{n-1}$.

To show the converse, note that $\mathrm{im}(\partial_i)$ is ${{n-2}\choose {i+1}}$-dimensional, so 
we may find up to  
${{n}\choose {i+1}}-{{n-2}\choose {i+1}}={{n-1}\choose  {i+1}} $
$i$-simplices in $\Delta_{n-1}^{(i)}$ whose removal does not alter the image of $\partial_i$. 
In particular, if
$r\le{{n-1}\choose {i+1}}$ then we may remove ${{n-1}\choose {i+1}}-r$ simplices of dimension  $i$, 
so that the remaining simplices 
form a Moore space of type $M(\ZZ^r,i)$.
We conclude that $M(\ZZ^r,i)$ can be represented by a subcomplex of $\Delta_{n-1}$ if, and only 
if ${{n-1}\choose {i+1}} \ge r$, 
which proves our claim.
\end{proof}

The theorem that we have just proved allows to improve some of our previous estimates. 
Let $M$ be a $(n-1)$-connected
closed $2n$-dimensional manifold. Up to homotopy type it can be built by attaching 
a $2n$-dimensional sphere to a wedge of 
$n$-dimensional spheres. Its LS-category is 3, so by Corollary \ref{cor:cat+conn} 
every combinatorial triangulation of $M$ has at least 
$$1+2n+(n-1)+\frac{1}{2}\cdot 3\cdot 2=3n+3$$
vertices. Similarly, Poincar\'e duality implies that there are cohomology classes 
in $H^n(M)$ whose product is non zero, so by Proposition \ref{prop:coho est} and 
Theorem \ref{thm:arithmetic} the covering type 
of $M$ is bounded by 
$$\ct(M)\ge \ct(n,n)=3n+3.$$
We can obtain better estimates by taking into account the rank of $H_n(M)$.

\begin{corollary}
\label{cor:n-1/2n mfd}
Let $M$ be a $(n-1)$-connected and $2n$-dimensional closed PL-manifold. 
Then every combinatorial triangulation
of $M$ has at least  $3n+3+k$ vertices, 
where $k$ is the minimal integer for which ${n+k+1 \choose k}\ge \mathrm{rank}\, H_n(M)$.
\end{corollary}
\begin{proof}
The first part of the argument is similar as in the proof of Corollary \ref{cor:cat+conn}. 
Let $K$ be a combinatorial
triangulation of $M$ and let $\UU$ be the cover of $M$ by open stars of vertices in $K$. 
Then there exist 
$2n+1$ vertices in $K$ that span a $2n$-dimensional simplex, and the union $U$ of corresponding 
open stars form 
the simplicial neighbourhood of that simplex. If we denote by $V$ the union of the remaining 
open stars, then 
the intersection $U\cap V$ is homeomorphic to $S^{2n-1}\times (0,1)$. 

If $n>1$ then the exactness of the homology sequence of the pair $(M,V)$ immediately implies 
that $V$ is 
a Moore space of type $M(\ZZ^r,n)$ where $r=\rank (H_n(M))$. 
If $n=1$, then we observe that $V$ can be deformed to a surface with boundary, and these are  
well-known to 
be homotopy equivalent to wedges of circles. In that case $V$ is a Moore space of type 
$M(\ZZ^r,1)$ where 
$r=\rank_{\ZZ_2}(H_1(M;\ZZ_2))$ (we use $\ZZ_2$-coefficients to obtain a statement that 
is valid for both 
orientable and non-orientable surfaces). 

By Theorem \ref{thm:ct wedge} $V$ is the union of at least $n+k+2$ open stars of vertices in $K$, where 
$k$ is the minimal integer for which ${n+k+1 \choose k}={n+k+1 \choose n+1}\ge r$. We conclude that $K$ 
has at least $(2n+1)+(n+k+2)=3n+3+k$ vertices, where $k$ is defined as above.
\end{proof}

%
%
%Our next aim is to improve lower bound for covering type of closed surfaces.
%\begin{prop}\label{ct for surfaces} Let $\mathcal{S}$ be
%a closed surface of genus ${\rm g} \geq 1$ (orientable or
%non-orientable). $$ \text{Then} \;\;\;\;\;\;\;\;\;\;
%\ct(\mathcal{S}) \;\geq \; \bigg[ \frac{9 +\sqrt{1+16 {\rm
%g}}}{2}\bigg]\;\;\;$$
%\end{prop}
%\begin{proof} Since $\hdim \mathcal{S} = 2$, there are $3$ sets in
%$\mathcal{U}$ with non-trivial intersection. The remaining part $V$
%(union of remaining sets) has $2{\rm g}$ linearly independent
%elements in $H^1(V)$ . So the estimate of \cite{K-W} gives
%$\ct(V)\geq \big[ \frac{3 +\sqrt{1+16 {\rm g}}}{2}\big]$. Adding
%this two inequalities we get the statement.
%\end{proof}.

\begin{corollary}\label{ct of product of two spheres}
\label{cor:prod 2 spheres}
Let $m\leq n$. Then we have
$$ \ct(S^m\times S^n) \;\geq \; m+2n+4$$
\end{corollary}
\begin{proof}
For $m<n$ the estimate follows by Theorem \ref{thm:arithmetic}, and for $m=n$ by the previous Corollary 
and the observation that $k=n+2$ is the minimal integer for which ${k \choose {n+1}}\ge \mathrm{rank}\, H_n(S^n\times S^n)=2$.
\end{proof}
We must add that the corresponding  estimate of number of vertices of a triangulation of a combinatorial manifold which is homotopy equivalent to $S^m \times S^n$ which follows from Corollary \ref{ct of product of two spheres} was shown in \cite{Bre-Kuh} (see also \cite{Bag-Dat}).
%\begin{lemma}
%Let $K$ be a subcomplex of $\Delta_n$. If $\widetilde H_{n-1}(K)\neq 0$, then $\widetilde{H}_i(K) = 0 $ for all $i < n-1$.
%\end{lemma}
%\begin{proof}
%The claim is obvious, because the only subcomplex of $\Delta_n$ that has non trivial homology in dimension $(n-1)$ is $\partial \Delta_n$.
%\end{proof}

Moore spaces for an arbitrary abelian group are usually constructed as wedges of Moore spaces of cyclic groups. It is therefore important 
to have estimates of the covering type of a wedges of spaces but one should expect some irregular behaviour. For example, by \cite[Proposition 4.1]{K-W}
we have the following relations 
$$ \ct(S^1) < \ct(S^1\vee S^1)= \ct(S^1\vee S^1\vee S^1) < \ct(S^1\vee S^1\vee S^1\vee S^1).$$

We can derive an upper estimate for the covering type of a wedge as follows. 
Let $m=\ct(X)$ and $n=\ct(Y)$. Then there are simplicial complexes
$K\le\Delta_{m-1}$ and $L\le\Delta_{n-1}$, such that $X\simeq |K|$ and $Y\simeq |L|$. Clearly, 
$K\vee L$ can be realized as a one-point union of $K$ and $L$ and is 
thus a subcomplex of $\Delta_{m+n-2}$. That estimate can be improved by gluing $K$ and $L$ 
along bigger subcomplexes. Indeed, if $\hdim(X)=k$ and $\hdim(Y)=l$, then 
$K$ and $L$ contain respectively a $k$-dimensional simplex $\sigma\le K$ and a $l$-dimensional 
simplex $\tau\le L$. If we assume that $k\le l$ and we glue together $K$ and $L$  
so that $\sigma$ is identified with a face of $\tau$, then the resulting complex is a subcomplex 
of $\Delta_{m+n-k-2}$, while its geometric realization is homotopy equivalent
to $X\vee Y$. Thus we have proved the following estimate.
\begin{proposition}\label{prop:ct wedge}
$$\ct(X \vee Y)\leq \ct(X) + \ct (Y) - \min \{\hdim X ,\, \hdim Y\}
-1 \,.$$
\end{proposition}

%A similar argument yields an estimate for the covering type of a connected sum of manifolds.
%\begin{proposition}\label{prop:ct connected sum}
%Let $M, \, N$ be $d$-dimensional manifolds. Then
%$$ \ct( M \# N) \leq \ct(M) + \ct(N) - (d+1)\,.$$
%\end{proposition}
%\begin{proof} As above, we find simplicial complexes $K$ with $\ct(M)$ vertices and $|K|\simeq M$, 
%and $L$ with 
%$\ct(N)$ vertices and $|L|\simeq N$. If we form the union of $K$ and $L$ along a common 
%$d$-dimensional simplex and
%excise the interior of the common simplex, we obtain a model for the connected sum $M\sharp N$. 
%The number of vertices 
%is as stated in the Corollary, because we have to subtract the vertices of the common simplex that 
%are counted twice.
%\end{proof}

Karoubi and Weibel \cite[Theorem 7.1]{K-W} have shown that the suspension of a space can increase 
its covering type at most by one.
However, it happens frequently that  the covering type drops after suspension. 

\begin{example}
We have shown in Corollary \ref{cor:prod 2 spheres} that $\ct(S^m\times S^n)\ge m+2n+4$. On the 
other hand, after suspension a product of spheres
splits as a wedge of spheres $\Sigma(S^m\times S^n)\simeq S^{m+1}\vee S^{n+1}\vee S^{m+n+1}$.
 Therefore, by Proposition \ref{prop:ct wedge}
$\ct\big(\Sigma(S^n\times S^n)\big)\le m+n+5$, so the covering type of the suspension of $S^m\times S^n$ 
is smaller at least by $n-1$ than the covering type 
of $S^m\times S^n$. Indeed, the result is not surprising at all if we have in mind our estimates 
based on the LS-category and the cohomology products,
and recall that the category of a suspension is always equal to 2, and that the cohomology products 
in a suspension  are always trivial.
\end{example}

Recall that we may construct Moore space whose homology is a finite cyclic group as follows. 
Let $X$ be the 2-dimensional CW-complex obtained by attaching a 2-cell to a circle by a map 
of degree $k$. 
The computation of cellular homology shows that $X$ is a Moore space of type $M(\ZZ_k,1)$.
Moreover, the $(i-1)$-fold suspension of $X$ is clearly a Moore space of type $M(\ZZ_k,i)$. 
To avoid repetition in the statements of the following two results we adopt the convention 
that $M(A,1)$ denotes the specific construction (wedge of circles and spaces of type $M(\ZZ_k,1)$ 
that we have just described).

\begin{proposition}
\label{prop:cyclic Moore}
$$i+3 \leq \ct(M(\bz_k, i)) \leq i+3k \,.$$
\end{proposition}
\begin{proof}
The lower estimate follows immediately from the observation  that $\hdim M(\ZZ_k,i)=i+1$. 
For the upper estimate recall that $M(\ZZ_k,1)$ is the mapping cone of a degree $k$ map between 
circles.
By \cite[Theorem 7.2]{K-W} and the fact that $\ct(S^1)=3$ we obtain $\ct(M(\ZZ_k,1))\le 1+3k$.
Since $M(\ZZ_k, i)$ can be obtained as a $(i-1)$-fold suspension of $M(\ZZ_k, 1)$ 
\cite[Theorem 7.1]{K-W} implies $\ct(M(\ZZ_k,1))\le i+3k$. 
\end{proof}
 
By combining Theorem \ref{thm:ct wedge}, Proposition \ref{prop:ct wedge} and 
Proposition \ref{prop:cyclic Moore} we obtain an upper
bound for the covering type of Moore spaces with  finitely generated homology. In the next 
result we will assume that there is at least some torsion
in homology, since the torsion-free case is settled by Theorem \ref{thm:ct wedge}. 

\begin{corollary} Assume $n \ge 1$. Then
$$\ct\big(M(\ZZ^r\oplus\ZZ_{k_1}\oplus\ldots\oplus \ZZ_{k_n},i)\big)\le i+k_0+
3(k_1+\ldots+k_n)-2(n-1), $$
where $k_0=\min\big\{k\ge 0\ \big|\  {i+k \choose i+1}\ge r\big\}$.
\end{corollary}
\begin{proof}
We are going to estimate the covering type of 
$$M(\ZZ^r,i)\vee M(\ZZ_{k_1},i)\vee\ldots\vee M(\ZZ_{k_n},i),$$
which can be taken as a model for $M(\ZZ^r\oplus\ZZ_{k_1}\oplus\ldots\oplus \ZZ_{k_n},i)$.
By Theorem \ref{thm:ct wedge} we have
$\ct(M(\ZZ^r,i))=i+k_0+1$. Since $M(\ZZ^r,i)$ is $i$-dimensional and $M(\ZZ_{k_1},i)$ is 
$(i+1)$-dimensional, 
Proposition \ref{prop:ct wedge} and Proposition \ref{prop:cyclic Moore} yield
$$\ct\big(M(\ZZ^r,i)\vee M(\ZZ_{k_1},i)\big)\le (i+k_0+1)+(i+3k_1)-i-1=i+k_0+3k_1,$$
If we add more finite cyclic summands both terms in the wedge are $(i+1)$-dimensional, 
and so the covering type increases 
at most by $3 k_j-2$ at each step. Observe that the formula is valid even if $r=0$, because in 
that case $k_0=0$.
\end{proof}

\section*{Acknowledgment}

We are very grateful to the referee for the careful reading of the paper and for
many valuable suggestions and comments which helped us to correct our arguments and 
improve considerably the manuscript.

\end{document}